\documentclass{article}
\usepackage{amscd}
\usepackage{amsmath}
\usepackage{amssymb}
\usepackage{amsthm}
\usepackage{ascmac}
\usepackage{bm}
\usepackage{graphicx}
\usepackage{here}
\usepackage[pdftex]{hyperref}%
\usepackage{mathtools}\numberwithin{equation}{section}
\usepackage{minitoc}
\usepackage{multicol}
\usepackage{multirow}
\usepackage{setspace}
\usepackage[dvipdfmx]{color}
\theoremstyle{definition}
\newtheorem{df}{Definition}

\newtheorem{q}{Problem}
\newtheorem{eg}{Example}
\newtheorem{rmk}{Remark}
\theoremstyle{plain}
\newtheorem{thm}{Theorem}
\newtheorem{prop}{Proposition}
\newtheorem{lem}{Lemma}
\newtheorem{cor}{Corollary}
\newenvironment{pf}{\begin{proof}}{\end{proof}}
\usepackage{tikz}
\usetikzlibrary{calc,intersections,patterns,through}
\newcommand{\ang}{\tikz@ang}
\def\tikz@ang(#1)(#2)#3{%
\pgfmathanglebetweenpoints{%
\pgfpointanchor{#1}{center}}{%
\pgfpointanchor{#2}{center}}
\pgfmathsetmacro{#3}{\pgfmathresult}%
}

\title{Algebraic Characterizations of Angle Multisections over Rings}
\author{Takashi HIROTSU}
\date{\today}
\begin{document}
\maketitle
\begin{abstract}
Let $n,$ $m \geq 2$ be integers, and let $R$ be a subring of $\mathbb R$ with field of fractions $F.$ 
In this article, we generalize the rational angle bisection problem previously proposed by the author to the following problem: which linearly independent vectors $\bm{a},$ $\bm{b} \in R^n$ form an angle with a sequence of $m$-sector vectors lying in $R^n$? 
When $\bm{a}$ and $\bm{b}$ are nonorthogonal, we prove that this condition is equivalent to the existence of a root in $F$ of a certain $m$-th degree polynomial over $R.$ 
In particular, when $R = \mathbb Z,$ the condition holds if and only if the polynomial has a root among the divisors of its constant term. 
When $m = 2^e$ with an integer $e \geq 1,$ we also prove that the condition is equivalent to $\cos (\theta /2^{e-1}) \in F,$ where $\theta$ is the angle between $\bm{a}$ and $\bm{b}.$
\end{abstract}
\section{Introduction}\label{sec-intro}
Throughout this article, let $n,$ $m \geq 2$ be integers, and let $R$ be a subring of  $\mathbb R$ with field of fractions $F.$ 
For any two nonzero vectors $\bm{a},$ $\bm{b} \in \mathbb R^n,$ we denote by $\angle(\bm{a},\bm{b})$ the angle between $\bm{a}$ and $\bm{b},$ as well as its measure in $[0,\pi ]$ depending on the context.
\begin{df}
Let $\bm{a},$ $\bm{b} \in R^n\setminus\{\bm{0}\}.$
\begin{enumerate}
\item[(1)]
A nonzero vector $\bm{c} \in R^n$ is called an {\it $i$-th $m$-sector vector of $\angle (\bm{a},\bm{b})$ over $R$} (or simply, an {\it angle multisector vector over $R$}), if there exist a sequence $(\bm{c}_0,\bm{c}_1,\dots, \bm{c}_m)$ of nonzero coplanar vectors in $R^n$ and an integer $r$ such that $\bm{c}_0 = \bm{a},$ $\bm{c}_m = \bm{b},$ $\bm{c}_i = \bm{c}$ for some $i \in \{ 1,\dots,m-1\},$ and
\[\angle(\bm{a},\bm{c}_j) = \frac{j}{m}(\angle (\bm{a},\bm{b})+2r\pi )\]
for each $j \in \{ 0,1,\dots,m\}.$ 
In this case, we also say that the $m+1$ vectors $\bm{c}_0,$ $\bm{c}_1,$ $\dots,$ $\bm{c}_m$ form a sequence of {\it $m$-equisector vectors of $\angle (\bm{a},\bm{b})$ over $R$} (or simply, {\it angle equisector vectors over $R$}).
\item[(2)]
The angle between $\bm{a}$ and $\bm{b}$ is said to be {\it $m$-sectable over $R$} if there exists a sequence of $m$-equisector vectors of $\angle (\bm{a},\bm{b})$ over $R.$
\end{enumerate}
\end{df}
\begin{eg}
Let $\bm{c}_0 = (7,1),$ $\bm{c}_1 = (2,1),$ $\bm{c}_2 = (1,1),$ $\bm{c}_3 = (1,2),$ and $\bm{c}_4 = (1,7).$ 
Then $\angle (\bm{c}_0,\bm{c}_4)$ has a sequence $(\bm{c}_1,\bm{c}_2,\bm{c}_3)$ of quadrisector vectors. 
That is, in the $xy$-plane, the acute angle between the two lines 
\[ y = \frac{1}{7}x \quad\text{and}\quad y = 7x\] 
has a sequence of quadrisector lines 
\[ y = \frac{1}{2}x, \quad y = x, \quad\text{and}\quad y = 2x\] 
as in Figure \ref{fig-multi-2d}.
\begin{figure}[h]
\centering
\includegraphics{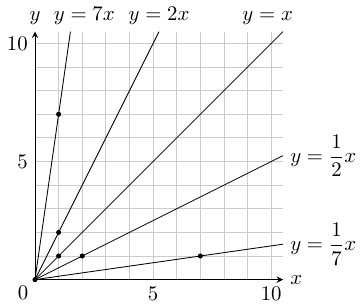}
\caption{A sequence of quadrisector lines of the acute angle between $y = (1/7)x$ and $y = 7x.$}\label{fig-multi-2d}
\end{figure}
\end{eg}
In this article, we generalize the rational angle bisection problem previously proposed by the author (see \cite[Problem 1]{Hir25}) to the following problem, which may be referred to as the {\it rational angle multisection problem}.
\begin{q}
Which linearly independent vectors $\bm{a},$ $\bm{b} \in R^n$ form an $m$-sectable angle over $R$?
\end{q}
\begin{rmk}
Let $\bm{a} = (1/\lambda )\bm{a}_0,$ $\bm{b} = (1/\lambda )\bm{b}_0 \in F^n$ with $\bm{a}_0,$ $\bm{b}_0 \in R^n$ and $\lambda \in R\setminus\{ 0\}.$ 
Then $\angle (\bm{a},\bm{b})$ is $m$-sectable over $F$ if and only if $\angle (\bm{a}_0,\bm{b}_0)$ is $m$-sectable over $R.$ 
\end{rmk}
The following theorem is known for the case when $m = 2.$ 
For any two vectors $\bm{a},$ $\bm{b} \in \mathbb R^n,$ we denote their inner product by $\langle\bm{a},\bm{b}\rangle.$ 
Let $F^{\times 2} = \{ q^2 \mid q \in F^\times\}.$
\begin{thm}[{\cite[Theorem 1]{Hir25}}]\label{thm-bisec}
Let $\bm{a},$ $\bm{b} \in R^n$ be linearly independent vectors. 
The following conditions are equivalent. 
\begin{enumerate}
\item[\textup{(B1)}]
A bisector vector $\bm{c}$ of $\angle(\bm{a},\bm{b})$ over $\mathbb R$ can be taken in $R^n.$
\item[\textup{(B2)}]
A normal vector $\bm{c}$ of a bisecting hyperplane between the two hyperplanes with normal vectors $\bm{a}$ and $\bm{b}$ can be taken in $R^n.$
\item[\textup{(B3)}]
There exists a vector $\bm{c} \in R^n$ such that 
\[\langle\bm{a},\bm{c}\rangle ^2|\bm{b}|^2 = \langle\bm{b},\bm{c}\rangle ^2|\bm{a}|^2.\]
\item[\textup{(B4)}]
We have 
\[ |\bm{a}|^2 \equiv |\bm{b}|^2 \pmod{F^{\times 2}}.\]
\end{enumerate}
Furthermore, if \textup{(B4)} holds, then $\bm{c}$ in \textup{(B1)}--\textup{(B3)} is parallel to 
\[\frac{|\bm{b}|}{\sqrt d}\bm{a}\pm\frac{|\bm{a}|}{\sqrt d}\bm{b},\] 
where $d$ is a positive number in $R$ such that $|\bm{a}|^2 \equiv |\bm{b}|^2 \equiv d \pmod{F^{\times 2}}.$
\end{thm}
In Section \ref{sec-multi}, we prove the following theorem.
\begin{thm}\label{thm-multi}
Let $\bm{a},$ $\bm{b} \in R^n$ be linearly independent and nonorthogonal vectors. 
Let 
\begin{align} 
p = \langle\bm{a},\bm{b}\rangle \quad\text{and}\quad s = \sqrt{|\bm{a}|^2|\bm{b}|^2-\langle\bm{a},\bm{b}\rangle ^2}. \label{eq-ps}  
\end{align} 
The following conditions are equivalent.
\begin{enumerate}
\item[\textup{(M1)}]
$\angle(\bm{a},\bm{b})$ is $m$-sectable over $R.$
\item[\textup{(M2)}]
We have 
\[\frac{1}{s}\tan\frac{\angle (\bm{a},\bm{b})}{m} \in F.\]
\item[\textup{(M3)}]
The $m$-th degree polynomial 
\begin{equation} 
\sum_{i = 0}^{\lfloor m/2\rfloor}(-s^2)^i\binom{m}{2i}t^{m-2i}-p\sum_{i = 0}^{\lfloor (m-1)/2\rfloor}(-s^2)^i\binom{m}{2i+1}t^{m-2i-1} \label{eq-deg-m}
\end{equation} 
in $t$ over $R$ has a root in $F.$
\end{enumerate}
Furthermore, if \eqref{eq-deg-m} has a root $t \in F,$ then a first $m$-sector vector of $\angle (\bm{a},\bm{b})$ over $R$ is given by 
\[\lambda ((t-p)\bm{a}+|\bm{a}|^2\bm{b})\] 
for some $\lambda \in R\setminus\{ 0\}.$
\end{thm}
For $m = 2,$ $3,$ $4,$ $5,$ $6,$ the polynomials in \eqref{eq-deg-m} are given by 
\begin{align} 
&t^2-2pt-s^2, \label{eq-deg-2} \\ 
&t^3-3pt^2-3s^2t+ps^2, \label{eq-deg-3} \\ 
&t^4-4pt^3-6s^2t^2+4ps^2t+s^4, \notag \\ 
&t^5-5pt^4-10s^2t^3+10ps^2t^2+5s^4t-ps^4, \notag \\ 
&t^6-6pt^5-15s^2t^4+20ps^2t^3+15s^4t^2-6ps^4t-s^6, \notag 
\end{align} 
respectively.\par
Combining Theorem \ref{thm-multi} with the rational root theorem, we immediately obtain the following result.
\begin{cor}
Let $F$ be an algebraic number field with class number $1,$ and let $R$ be its ring of integers. 
Let $\bm{a},$ $\bm{b} \in R^n$ be linearly independent and nonorthogonal vectors. 
Then $\angle (\bm{a},\bm{b})$ is $m$-sectable over $R$ if and only if \eqref{eq-deg-m} has a root in $R.$ 
The possible roots of \eqref{eq-deg-m} in $R$ are restricted to the divisors of $s^m$ if $m$ is even, or those of $ps^{m-1}$ if $m$ is odd.
\end{cor}
\begin{eg}
Let $\bm{a} = (1,1),$ $\bm{b} = (-2,11) \in \mathbb Z^2.$ 
Then $\angle (\bm{a},\bm{b})$ is trisectable over $\mathbb Z,$ since the cubic polynomial $t^3-27t^2-507t+1521$ corresponding to \eqref{eq-deg-3} has an integral root $t = 39.$ 
A first trisector vector of $\angle (\bm{a},\bm{b})$ over $\mathbb Z$ is given by 
\[ (t-\langle\bm{a},\bm{b}\rangle )\bm{a}+|\bm{a}|^2\bm{b} = (39-9)(1,1)+2(-2,11) = 26(1,2).\]
\end{eg}
We also discuss the $m$-sectability of right angles in Section \ref{sec-multi}.\par
In Theorem \ref{thm-bisec}, (B1) is equivalent to $\cos\angle (\bm{a},\bm{b}) \in F.$ 
In Section \ref{sec-pw2}, we generalize this equivalence to the case when $m = 2^e$ with an integer $e \geq 1.$
\begin{thm}\label{thm-pw2}
Let $\bm{a},$ $\bm{b} \in R^n\setminus\{\bm{0}\},$ and let $e \geq 1$ be an integer. 
The following conditions are equivalent.
\begin{enumerate}
\item[\textup{(C1)}]
$\angle(\bm{a},\bm{b})$ is $2^e$-sectable over $R.$
\item[\textup{(C2)}]
We have 
\[\cos\frac{\angle (\bm{a},\bm{b})}{2^{e-1}} \in F.\]
\item[\textup{(C3)}]
For each $i \in \{ 0,1,\dots,e-1\},$ we have 
\[\cos\frac{\angle (\bm{a},\bm{b})}{2^i} \in F.\]
\end{enumerate}
\end{thm}
Theorem \ref{thm-pw2} enables us to verify the $2^e$-sectability of $\angle (\bm{a},\bm{b})$ over $R$ through purely algebraic calculations by the double angle formula.
\begin{eg}
Let $\bm{a} = (1,1),$ $\bm{b} = (-17,31) \in \mathbb Z^2.$ 
Then $\angle (\bm{a},\bm{b})$ is quadrisectable over $\mathbb Z,$ since 
\[\cos\angle (\bm{a},\bm{b}) = \frac{\langle\bm{a},\bm{b}\rangle}{|\bm{a}||\bm{b}|} = \frac{14}{\sqrt 2\cdot 25\sqrt 2} = \frac{7}{25}\] 
and 
\[\cos\frac{\angle (\bm{a},\bm{b})}{2} = \sqrt{\frac{1+\cos\angle (\bm{a},\bm{b})}{2}} = \sqrt{\frac{1}{2}\left(1+\frac{7}{25}\right)} = \sqrt{\frac{16}{25}} = \frac{4}{5} \in \mathbb Q.\]
\end{eg}
\section{Sequences of Angle Equisectors}\label{sec-ext}
In this section, we discuss the construction and extension of sequences of angle equisector vectors derived from Householder reflections (see \cite[Section 1.1]{Hum90}). 
We also discuss the decomposition of angle multisectability.
\subsection{Construction of Angle Equisectors via Householder Reflections}\label{subsec-const}
The following proposition is fundamental and plays an important role in practical applications, as it enables us to extend a sequence of angle equisector vectors over $R$ to an arbitrary length (see Subsection \ref{subsec-ext}).
\begin{prop}\label{prop-ref}
Let $\bm{a},$ $\bm{c} \in R^n$ be linearly independent vectors. 
Then there exists a nonzero vector $\bm{b} \in R\bm{a}+R\bm{c}$ such that $\bm{c}$ is a bisector vector of $\angle (\bm{a},\bm{b}).$ 
Such a vector $\bm{b}$ is parallel to $2\langle\bm{a},\bm{c}\rangle\bm{c}-|\bm{c}|^2\bm{a}.$
\end{prop}
\begin{pf}
Let $\bm{b} = 2\langle\bm{a},\bm{c}\rangle\bm{c}-|\bm{c}|^2\bm{a}.$ 
It suffices to show that $\angle (\bm{a},\bm{c}) = \angle (\bm{b},\bm{c}).$\par
If $\langle\bm{a},\bm{c}\rangle = 0,$ then $\angle (\bm{a},\bm{c}) = \angle (\bm{b},\bm{c}) = \pi /2,$ since $\bm{b} = -|\bm{c}|^2\bm{a}.$\par
Suppose that $\langle\bm{a},\bm{c}\rangle \neq 0.$ 
Then 
\[ |\bm{b}|^2 = 4\langle\bm{a},\bm{c}\rangle ^2|\bm{c}|^2-4\langle\bm{a},\bm{c}\rangle ^2|\bm{c}|^2+|\bm{c}|^4|\bm{a}|^2 = |\bm{c}|^4|\bm{a}|^2,\] 
which implies $|\bm{b}| = |\bm{c}|^2|\bm{a}|.$ 
Therefore, a bisector vector of $\angle (\bm{a},\bm{b})$ is given by 
\[\bm{a}+\frac{|\bm{a}|}{|\bm{b}|}\bm{b} = \bm{a}+\frac{1}{|\bm{c}|^2}(2\langle\bm{a},\bm{c}\rangle\bm{c}-|\bm{c}|^2\bm{a}) = \frac{2\langle\bm{a},\bm{c}\rangle}{|\bm{c}|^2}\bm{c},\] 
which is parallel to $\bm{c}.$ 
Thus, we conclude that $\angle (\bm{a},\bm{c}) = \angle (\bm{b},\bm{c}).$
\end{pf}
\begin{eg}
Let $\bm{c}_2 = (1,1)$ and $\bm{c}_3 = (1,2).$  
Then 
\[ 2\langle\bm{c}_2,\bm{c}_3\rangle\bm{c}_3-|\bm{c}_3|^2\bm{c}_2 = 2\cdot 3(1,2)-5(1,1) = (1,7).\] 
Let $\bm{c}_4 = (1,7).$ 
Then $\bm{c}_3$ is a bisector vector of $\angle (\bm{c}_2,\bm{c}_4).$ 
Furthermore, let $\bm{c}_0 = (7,1)$ and $\bm{c}_1 = (2,1).$ 
Then $\bm{c}_1$ is a bisector vector of $\angle (\bm{c}_0,\bm{c}_2),$ and $\bm{c}_2$ is a bisector vector of $\angle (\bm{c}_1,\bm{c}_3),$ since the points $(1,2)$ and $(1,7)$ are symmetric to $(2,1)$ and $(7,1),$ respectively, across the line $x = y$ with direction vector $\bm{c}_2.$
Thus, $\angle (\bm{c}_0,\bm{c}_4)$ has a sequence $(\bm{c}_1,\bm{c}_2,\bm{c}_3)$ of quadrisector vectors.
\end{eg}
\begin{eg}
Let $\bm{c}_2 = (1,1,1)$ and $\bm{c}_3 = (1,2,3).$ 
Then 
\[ 2\langle\bm{c}_2,\bm{c}_3\rangle\bm{c}_3-|\bm{c}_3|^2\bm{c}_2 = 2\cdot 6(1,2,3)-14(1,1,1) = 2(-1,5,11).\] 
Let $\bm{c}_4 = (-1,5,11).$ 
Then $\bm{c}_3$ is a bisector vector of $\angle (\bm{c}_2,\bm{c}_4).$ 
Furthermore, let $\bm{c}_0 = (11,5,-1)$ and $\bm{c}_1 = (3,2,1).$ 
Then $\bm{c}_1$ is a bisector vector of $\angle (\bm{c}_0,\bm{c}_2),$ and $\bm{c}_2$ is a bisector vector of $\angle (\bm{c}_1,\bm{c}_3),$ since the points $(1,2,3)$ and $(-1,5,11)$ are symmetric to $(3,2,1)$ and $(11,5,-1),$ respectively, across the line $x = y = z$ with direction vector $\bm{c}_2.$ 
Thus, $\angle (\bm{c}_0,\bm{c}_4)$ has a sequence $(\bm{c}_1,\bm{c}_2,\bm{c}_3)$ of quadrisector vectors. 
That is, in the $xyz$-space, the acute angle between the two lines 
\[\frac{x}{11} = \frac{y}{5} = -z \quad\text{and}\quad -x = \frac{y}{5} = \frac{z}{11}\] 
has a sequence of quadrisector lines 
\[\frac{x}{3} = \frac{y}{2} = z, \quad x = y = z, \quad\text{and}\quad x = \frac{y}{2} = \frac{z}{3}.\]
\end{eg}
\subsection{Recursive Extensions of Sequences of Angle Equisectors}\label{subsec-ext}
Proposition \ref{prop-ref} immediately yields the following three corollaries.
\begin{cor}\label{cor-infty}
For any nonzero vectors $\bm{c}_0,$ $\bm{c}_1 \in R^n,$ there exists an infinite sequence $(\bm{c}_i)_{i \geq 0}$ of nonzero vectors in $R\bm{c}_0+R\bm{c}_1$ such that 
\[\angle (\bm{c}_0,\bm{c}_1) = \angle (\bm{c}_1,\bm{c}_2) = \cdots = \angle (\bm{c}_{i-1},\bm{c}_i) = \cdots.\]
\end{cor}
\begin{cor}\label{cor-fin}
A given sequence of angle equisector vectors over $R$ can be extended to an arbitrary length. 
Specifically, for any integers $m_1$ and $m_2$ such that $2 \leq m_1 < m_2,$ any sequence of angle $m_1$-equisector vectors over $R$ can be extended to a sequence of angle $m_2$-equisector vectors over $R$ by appending $m_2-m_1$ suitable vectors in $R^n.$
\end{cor}
\begin{cor}\label{cor-multi}
Let $\bm{a},$ $\bm{b} \in R^n\setminus\{\bm{0}\}.$ 
In a sequence of $m$-equisector vectors $(\bm{c}_0,\bm{c}_1,\dots,\bm{c}_m)$ of $\angle (\bm{a},\bm{b})$ over $\mathbb R,$ if $\bm{c}_{i-1}$ and $\bm{c}_i$ can be taken in $R^n$ for some $i \in \{ 1,\dots,m\},$ then $\angle (\bm{a},\bm{b})$ is $m$-sectable over $R.$
\end{cor}
Corollary \ref{cor-multi} is used in the proofs of Theorems \ref{thm-multi} and \ref{thm-pw2}.
\subsection{Decomposition of Angle Multisectability}\label{subsec-decomp}
The construction of sequences of angle equisector vectors described above leads to the following result.
\begin{prop}\label{prop-decomp}
Let $\bm{a},$ $\bm{b} \in R^n\setminus\{\bm{0}\}.$ 
Let $m_1,$ $\dots,$ $m_r \geq 2$ be pairwise coprime integers, and let $m = m_1\cdots m_r.$ 
Then the following conditions are equivalent.
\begin{enumerate}
\item[{\rm (i)}]
$\angle (\bm{a},\bm{b})$ is $m$-sectable over $R.$
\item[{\rm (ii)}]
$\angle (\bm{a},\bm{b})$ is $m_i$-sectable over $R$ for each $i \in \{ 1,\dots,r\}.$
\end{enumerate}
\end{prop}
\begin{pf}
It is obvious that (i) implies (ii). 
To show that (ii) implies (i), it suffices to consider the case when $r = 2.$ 
Suppose that $\angle (\bm{a},\bm{b})$ is both $m_1$-sectable and $m_2$-sectable over $R.$ 
Since $m_1$ and $m_2$ are coprime, there exist integers $q_1,$ $q_2 \neq 0$ such that $m_1q_1+m_2q_2 = 1.$ 
Then a $|q_2|$-th $m_1$-sector vector and a $|q_1|$-th $m_2$-sector vector of $\angle (\bm{a},\bm{b})$ over $\mathbb R$ can be taken in $R^n,$ which are an $m_2|q_2|$-th $m_1m_2$-sector vector and an $m_1|q_1|$-th $m_1m_2$-sector vector over $\mathbb R,$ respectively. 
This implies that (i) holds by Corollary \ref{cor-multi}, since $\big| m_1|q_1|-m_2|q_2|\big| = 1.$
\end{pf}
\section{Angle Multisections}\label{sec-multi}
In this section, we provide a proof and examples of Theorem \ref{thm-multi}.
\subsection{Orthonormal Basis for the Plane Spanned by Two Vectors}\label{subsec-lem}
In the following arguments, we use the following change of basis.
\begin{lem}\label{lem-1}
Let $\bm{a},$ $\bm{b} \in \mathbb R^n$ be linearly independent vectors, and let $p$ and $s$ be defined as in \eqref{eq-ps}. 
Then 
\begin{align} 
\bm{v}_1 = \frac{\bm{a}}{|\bm{a}|} \quad\text{and}\quad \bm{v}_2 = \frac{|\bm{a}|}{s}\bm{b}-\frac{p}{s|\bm{a}|}\bm{a} \label{eq-v1-v2} 
\end{align} 
satisfy 
\begin{align} 
&|\bm{v}_1| = |\bm{v}_2| = 1, \quad \bm{v}_1\perp\bm{v}_2, \label{eq-v12} \\ 
&\bm{a} =|\bm{a}|\bm{v}_1, \quad\text{and}\quad \bm{b} = \frac{p}{|\bm{a}|}\bm{v}_1+\frac{s}{|\bm{a}|}\bm{v}_2. \label{eq-ab} 
\end{align} 
\end{lem}
\begin{pf}
These relations follow directly from the definition of the inner product, 
\[ p = |\bm{a}||\bm{b}|\cos\angle (\bm{a},\bm{b}),\] 
and the formula for the area of the parallelogram spanned by $\bm{a}$ and $\bm{b},$ 
\[ s = |\bm{a}||\bm{b}|\sin\angle (\bm{a},\bm{b}).\qedhere\]  
\end{pf}
\subsection{Multisections of Non-right Angles}\label{subsec-pf}
We also use the following two lemmas.
\begin{lem}\label{lem-2}
Let $\bm{a},$ $\bm{b} \in R^n$ be linearly independent vectors, and let $s$ be defined as in \eqref{eq-ps}. 
If $\bm{c} \in (R\bm{a}+R\bm{b})\setminus\{\bm{0}\}$ and $\bm{a}$ and $\bm{c}$ are nonorthogonal, then 
\begin{equation} 
\frac{1}{s}\tan\angle (\bm{a},\bm{c}) \in F. \label{eq-tan} 
\end{equation}
\end{lem}
\begin{pf}
Let $\bm{c} \in (R\bm{a}+R\bm{b})\setminus\{\bm{0}\}$ and $\varphi = \angle (\bm{a},\bm{c}).$ 
Then there exist scalars $\lambda, \mu \in R,$ not both zero, such that $\bm{c} = \lambda\bm{a}+\mu\bm{b}.$ 
We use the orthonormal basis $\{\bm{v}_1, \bm{v}_2\}$ of $\mathbb R\bm{a}+\mathbb R\bm{b}$ defined in \eqref{eq-v1-v2}.
On one hand, by \eqref{eq-v12}, $\bm{c}$ can be expressed as 
\[\bm{c} = |\bm{c}|(\cos\varphi )\bm{v}_1+|\bm{c}|(\sin\varphi )\bm{v}_2.\] 
On the other hand, substituting \eqref{eq-ab} into the relation $\bm{c} = \lambda \bm{a} + \mu \bm{b}$, we obtain 
\[\bm{c} = \lambda |\bm{a}|\bm{v}_1+\mu\left(\frac{p}{|\bm{a}|}\bm{v}_1+\frac{s}{|\bm{a}|}\bm{v}_2\right) = \frac{\lambda |\bm{a}|^2+\mu p}{|\bm{a}|}\bm{v}_1+\frac{\mu s}{|\bm{a}|}\bm{v}_2.\] 
Since $\bm{v}_1$ and $\bm{v}_2$ are linearly independent over $\mathbb R,$ we can equate the coefficients to obtain 
\[ |\bm{c}|\cos\varphi = \frac{\lambda |\bm{a}|^2+\mu p}{|\bm{a}|} \quad\text{and}\quad |\bm{c}|\sin\varphi = \frac{\mu s}{|\bm{a}|}.\] 
Dividing the second equation by the first yields 
\[\tan\varphi = \frac{\mu s}{\lambda |\bm{a}|^2+\mu p},\] 
which implies \eqref{eq-tan} since $\mu,$ $\lambda |\bm{a}|^2+\mu p \in R.$
\end{pf}
\begin{lem}\label{lem-3}
Let $\bm{a},$ $\bm{b} \in \mathbb R^n$ be linearly independent and nonorthogonal vectors, and let $s$ be defined as in \eqref{eq-ps}. 
Then \eqref{eq-deg-m} has $m$ distinct nonzero real roots. 
Every root $t$ of \eqref{eq-deg-m} corresponds uniquely to a real number $\varphi \in (-\pi/2, \pi/2)$ satisfying $\angle (\bm{a},\bm{b}) \equiv m\varphi \pmod\pi$ via the relation 
\[ t = \frac{s}{\tan\varphi}.\]
\end{lem}
\begin{pf}
Let $\theta = \angle (\bm{a},\bm{b}).$ 
Suppose that $\varphi \in (-\pi/2, \pi/2)$ satisfies $\theta \equiv m\varphi \pmod\pi.$ 
Since $\bm{a}$ and $\bm{b}$ are linearly independent and nonorthogonal, $\theta \neq 0,$ $\pi /2,$ $\pi,$ which implies $\varphi \neq 0$ and $\tan \varphi$ is well-defined.
Let $u = \tan\varphi,$ where $u \neq 0.$ 
Substituting 
\[\tan\theta = \frac{s}{p}\] 
and 
\[\tan m\varphi = \frac{\displaystyle\sum_{i = 0}^{\lfloor (m-1)/2\rfloor}(-1)^i\binom{m}{2i+1}u^{2i+1}}{\displaystyle\sum_{i = 0}^{\lfloor m/2\rfloor}(-1)^i\binom{m}{2i}u^{2i}}\] 
into $\tan\theta = \tan m\varphi,$ we obtain 
\begin{equation} 
s\sum_{i = 0}^{\lfloor m/2\rfloor}(-1)^i\binom{m}{2i}u^{2i} = p\sum_{i = 0}^{\lfloor (m-1)/2\rfloor}(-1)^i\binom{m}{2i+1}u^{2i+1}. \label{eq-deg-m-prime} 
\end{equation} 
Let $t = s/u.$ 
Substituting $u = s/t$ into \eqref{eq-deg-m-prime}, we obtain 
\[ s\sum_{i = 0}^{\lfloor m/2\rfloor}(-s^2)^i\binom{m}{2i}t^{-2i} = ps\sum_{i = 0}^{\lfloor (m-1)/2\rfloor}(-s^2)^i\binom{m}{2i+1}t^{-2i-1}.\] 
Multiplying both sides by $t^m/s,$ we find that $t$ is a root of \eqref{eq-deg-m}. 
Since there exist $m$ distinct values of $\tan \varphi$ corresponding to the $m$ solutions of $\theta \equiv m\varphi \pmod\pi$ within $(-\pi/2,\pi/2),$ we conclude that \eqref{eq-deg-m} has $m$ distinct nonzero real roots.
\end{pf}
We are now ready to prove Theorem \ref{thm-multi}.
\begin{pf}[Proof of Theorem \ref{thm-multi}]
By Lemma \ref{lem-2}, (M1) implies (M2).\par
Furthermore, (M2) implies (M3) as a direct consequence of Lemma \ref{lem-3}.\par
It remains to show that (M3) implies (M1) and the last assertion. 
Let $t \in F$ be a root of \eqref{eq-deg-m}, and let $\varphi = \arctan (s/t).$ 
We define $\bm{v}_1$ and $\bm{v}_2$ as in \eqref{eq-v1-v2}. 
Then $s/\tan\varphi = t,$ $\angle (\bm{a},\bm{b}) \equiv m\varphi \pmod\pi$ by Lemma \ref{lem-3}, and therefore $\angle (\bm{a},\bm{b})$ has a first $m$-sector vector $\bm{c}$ over $\mathbb R$ given by 
\begin{align*} 
\bm{c} &= \frac{s|\bm{a}|}{\sin\varphi}((\cos\varphi )\bm{v}_1+(\sin\varphi )\bm{v}_2) \\ 
&= \frac{s}{\tan\varphi}\bm{a}+s|\bm{a}|\left(\frac{|\bm{a}|}{s}\bm{b}-\frac{p}{s|\bm{a}|}\bm{a}\right) \\ 
&= \left(\frac{s}{\tan\varphi}-p\right) \bm{a}+|\bm{a}|^2\bm{b} \\ 
&= (t-p)\bm{a}+|\bm{a}|^2\bm{b}, 
\end{align*} 
which lies in $F^n.$ 
Since $t \in F,$ there exists a scalar $\lambda \in R\setminus\{ 0\}$ such that $\lambda t \in R,$ which implies $\lambda\bm{c} \in R^n.$ 
Thus, by Corollary \ref{cor-multi}, $\angle (\bm{a},\bm{b})$ is $m$-sectable over $R.$
\end{pf}
\begin{eg}
Since the discriminant of the quadratic equation $t^2-2pt-s^2 = 0$ is given by 
\[ (-2p)^2-4\cdot 1\cdot (-s^2) = 4(p^2+s^2) = 4|\bm{a}|^2|\bm{b}|^2,\] 
the quadratic polynomial \eqref{eq-deg-2} has a root in $F$ if and only if $|\bm{a}|^2|\bm{b}|^2 \in F^{\times 2},$ which is equivalent to (B4) in Theorem \ref{thm-bisec}. 
Consequently, $\angle (\bm{a},\bm{b})$ is bisectable over $R$ if and only if (B4) holds.
\end{eg}
\begin{eg}
Let $\bm{a} = (1,1,1),$ $\bm{b} = (-11,6,23) \in \mathbb Z^3.$ 
Then $\angle (\bm{a},\bm{b})$ is trisectable over $\mathbb Z,$ since the cubic polynomial $t^3-54t^2-5202t+31212$ corresponding to \eqref{eq-deg-3} has an integral root $t = 102.$ 
A first trisector vector of $\angle (\bm{a},\bm{b})$ over $\mathbb Z$ is given by 
\[ (t-\langle\bm{a},\bm{b}\rangle )\bm{a}+|\bm{a}|^2\bm{b} = (102-18)(1,1,1)+3(-11,6,23) = 51(1,2,3).\]
\end{eg}
\subsection{Multisections of Right Angles}\label{subsec-pf}
In this subsection, we consider the $m$-sectability of right angles.
\begin{prop}
The following conditions are equivalent.
\begin{enumerate}
\item[\textup{(i)}]
For any perpendicular vectors $\bm{a},$ $\bm{b} \in R^n\setminus\{\bm{0}\},$ $\angle (\bm{a},\bm{b})$ is $m$-sectable over $R.$ 
\item[\textup{(ii)}]
We have 
\[\tan\frac{\pi}{2m} \in F.\]
\end{enumerate}
\end{prop}
\begin{pf}
Suppose that (i) holds. 
Let $\bm{a} = (1,0,0,\dots,0)$ and $\bm{b} = (0,1,0,\dots,0).$ 
Then a first $m$-sector vector $\bm{c}$ of $\angle (\bm{a},\bm{b})$ over $R$ is given by 
\begin{equation} 
\bm{c} = \lambda\left(\bm{a}+\tan\frac{\pi}{2m}\bm{b}\right) \quad\text{or}\quad \bm{c} = \lambda\left(\cot\frac{\pi}{2m}\bm{a}+\bm{b}\right) \label{eq-right} 
\end{equation} 
for some $\lambda \in R\setminus\{ 0\}.$ 
Since $\bm{c} \in R^n,$ and 
\[\bm{c} = \left(\lambda,\lambda\tan\frac{\pi}{2m},0,\dots,0\right) \quad\text{or}\quad \bm{c} = \left(\lambda\cot\frac{\pi}{2m},\lambda,0,\dots,0\right),\] 
we conclude that (ii) holds.\par
Conversely, suppose that (ii) holds. 
Let $\bm{a},$ $\bm{b} \in R^n\setminus\{\bm{0}\}$ be perpendicular vectors. 
Then a first $m$-sector vector $\bm{c}$ of $\angle (\bm{a},\bm{b})$ over $\mathbb R$ can be taken in $R^n$ by letting \eqref{eq-right} for some $\lambda \in R\setminus\{ 0\}.$ 
This implies (i) by Corollary \ref{cor-multi}.\qedhere
\end{pf}
\begin{eg}
\begin{enumerate}
\item[(1)]
For any perpendicular vectors $\bm{a},$ $\bm{b} \in R^n\setminus\{\bm{0}\},$ $\angle (\bm{a},\bm{b})$ is bisectable over $R,$ since $\tan (\pi /4) = 1 \in F.$
\item[(2)]
If $\sqrt 3 \in F,$ then, for any perpendicular vectors $\bm{a},$ $\bm{b} \in R^n\setminus\{\bm{0}\},$ $\angle (\bm{a},\bm{b})$ is trisectable and sexisectable over $R,$ since $\tan (\pi /6) = 1/\sqrt 3$ and Proposition \ref{prop-decomp} holds. 
The converse is also true.
\end{enumerate}
\end{eg}
\section{Angle $2^e$-sections}\label{sec-pw2}
In this section, we provide a proof and examples of Theorem \ref{thm-pw2}. 
\begin{pf}[Proof of Theorem \ref{thm-pw2}]
Let $\theta = \angle (\bm{a},\bm{b}).$\par
For the case when $e = 1,$ note that (B1) in Theorem \ref{thm-bisec} is equivalent to $\cos\theta \in F.$ 
Indeed, by Theorem \ref{thm-bisec}, (B1) is equivalent to (B4), which can be restated as $|\bm{a}||\bm{b}| \in F^\times,$ or equivalently, $\cos\theta \in F$ since $\langle\bm{a},\bm{b}\rangle = |\bm{a}||\bm{b}|\cos\theta$ and $\langle\bm{a},\bm{b}\rangle \in R.$\par
We now consider the general case when $e \geq 1.$ 
Suppose that (C1) holds. 
In a sequence $(\bm{c}_0,\bm{c}_1,\dots,\bm{c}_m)$ of $2^e$-equisector vectors of $\angle (\bm{a},\bm{b})$ over $R,$ the vector $\bm{c}_1$ is a bisector vector of $\angle (\bm{a},\bm{c}_2)$ over $R,$ which implies (C2) by the above argument.\par
Next, (C2) implies (C3), since 
\[\cos ^2\frac{\angle (\bm{a},\bm{b})}{2^i} = \frac{1}{2}\left( 1+\cos\frac{\angle (\bm{a},\bm{b})}{2^{i-1}}\right) \quad (i \in \{ 1,\dots,e-1\})\] 
by the double-angle formula.\par
Finally, suppose that (C3) holds. 
First, $\cos\theta \in F$ implies that a bisector vector $\bm{c}_{2^{e-1}}$ of $\angle (\bm{a},\bm{b})$ over $\mathbb R$ can be taken in $R^n.$ 
Second, $\cos (\theta /2) \in F$ implies that a bisector vector $\bm{c}_{2^{e-2}}$ of $\angle (\bm{a},\bm{c}_{2^{e-1}})$ over $\mathbb R$ can be taken in $R^n,$ which serves as a first quadrisector vector of $\angle (\bm{a},\bm{b}).$ 
By repeating this process $e$ times, we find that a first $2^e$-sector vector $\bm{c}_1$ of $\angle (\bm{a},\bm{b})$ over $\mathbb R$ can be taken in $R^n.$ 
This implies (C1) by Corollary \ref{cor-multi}.
\end{pf}
\begin{eg}
Let $\bm{a} = (1,1,1),$ $\bm{b} = (-59,1,61) \in \mathbb Z^3.$ 
Then $\angle (\bm{a},\bm{b})$ is quadrisectable over $\mathbb Z,$ since 
\[\cos\angle (\bm{a},\bm{b}) = \frac{\langle\bm{a},\bm{b}\rangle}{|\bm{a}||\bm{b}|} = \frac{3}{\sqrt 3\cdot 49\sqrt 3} = \frac{1}{49}\] 
and 
\[\cos\frac{\angle (\bm{a},\bm{b})}{2} = \sqrt{\frac{1+\cos\angle (\bm{a},\bm{b})}{2}} = \sqrt{\frac{1}{2}\left(1+\frac{1}{49}\right)} = \sqrt{\frac{25}{49}} = \frac{5}{7} \in \mathbb Q.\]
\end{eg}
\begin{eg}
Theorem \ref{thm-pw2} remains valid for the case when $\angle (\bm{a},\bm{b}) = \pi /2.$ 
For example, if $\bm{a} = (a_1,a_2),$ $\bm{b} = (-a_2,a_1) \in R^2,$ then $\cos\angle (\bm{a},\bm{b}) = 0 \in F$ and $\angle (\bm{a},\bm{b})$ has a bisector vector $(a_1-a_2,a_1+a_2) \in R^2,$ which is consistent with the fact that $|\bm{a}|^2 = |\bm{b}|^2$ implies (B4).
\end{eg}

\end{document}